\documentclass[a4paper,12pt]{article}
\usepackage{tikz}
\usetikzlibrary{matrix,arrows,decorations.markings}
\usepackage{amsmath,amsfonts,amssymb,amsthm,url,textcomp}
\usepackage{csquotes}
\usepackage{a4wide}
\usepackage[numbers]{natbib}
\usepackage{fancyhdr}
\usepackage{anyfontsize}
\usepackage{hyperref}
\usepackage{hhline}
\usepackage{leftidx}
\usepackage{mdframed}
\usepackage{enumitem}

\allowdisplaybreaks

\newtheorem{theorem}{Theorem}\numberwithin{theorem}{section}
\newtheorem{lemma}[theorem]{Lemma}

\newtheorem{notation}[theorem]{Notation}

\newtheorem{question}[theorem]{Question}

\newtheorem{theoremm}{Theorem}\numberwithin{theoremm}{subsection}
\newtheorem{lemmma}[theoremm]{Lemma}

\newtheorem{propposition}[theoremm]{Proposition}

\numberwithin{theoremmm}{subsubsection}

\theoremstyle{remark}

\newtheorem{remmark}[theoremm]{Remark}

\theoremstyle{definition}
\newtheorem{deffinition}[theoremm]{Definition}
\newtheorem{definition}[theorem]{Definition}

\newcommand{\Rad}{\operatorname{Rad}}
\newcommand{\Aut}{\operatorname{Aut}}
\newcommand{\Alt}{\operatorname{Alt}}

\newcommand{\ord}{\operatorname{ord}}
\newcommand{\Sym}{\operatorname{Sym}}

\newcommand{\C}{\operatorname{C}}

\newcommand{\Inn}{\operatorname{Inn}}
\newcommand{\id}{\operatorname{id}}
\newcommand{\N}{\operatorname{N}}

\newcommand{\GL}{\operatorname{GL}}

\newcommand{\D}{\operatorname{D}}

\newcommand{\IN}{\mathbb{N}}

\newcommand{\perm}{\mathrm{perm}}

\newcommand{\Stab}{\operatorname{Stab}}

\newcommand{\IZ}{\mathbb{Z}}

\newcommand{\maol}{\operatorname{maol}}

\newcommand{\Exp}{\operatorname{Exp}}

\newcommand{\cent}{\mathrm{cent}}

\newcommand{\reg}{\mathrm{reg}}

\newcommand{\ffrak}{\mathfrak{f}}
\newcommand{\pow}{\mathrm{pow}}
\newcommand{\Ocal}{\mathcal{O}}

\begin{document}

\title{Finite transitive permutation groups with \\only small normaliser orbits}

\author{Alexander Bors and Michael Giudici\thanks{First author's address: Johann Radon Institute for Computational and Applied Mathematics (RICAM), Altenberger Stra{\ss}e 69, 4040 Linz, Austria. E-mail: \href{mailto:alexander.bors@ricam.oeaw.ac.at}{alexander.bors@ricam.oeaw.ac.at} \newline Second author's address: The University of Western Australia, Centre for the Mathematics of Symmetry and Computation, 35 Stirling Highway, Crawley 6009, WA, Australia. E-mail: \href{mailto:michael.giudici@uwa.edu.au}{michael.giudici@uwa.edu.au} \newline The first author is supported by the Austrian Science Fund (FWF), project J4072-N32 \enquote{Affine maps on finite groups}. The second author was supported by the Australian Research Council Discovery Project DP200101951.  This research arose out of the annual CMSC research retreat and the authors thank the other participants for making it such an enjoyable event.  \newline 2020 \emph{Mathematics Subject Classification}: Primary 20B10. Secondary 20D45, 20F28. \newline \emph{Key words and phrases:} Transitive permutation group, Normaliser, Normaliser orbit.}}

\date{\today}

\maketitle

\abstract{We study finite transitive permutation groups $G\leqslant\Sym(\Omega)$ such that all orbits of the conjugation action on $G$ of the normaliser of $G$ in $\Sym(\Omega)$ have size bounded by some constant. Our results extend recent results, due to the first author, on finite abstract groups $G$ such that all orbits of the natural action of the automorphism group $\Aut(G)$ on $G$ have size bounded by some constant.}

\section{Introduction and main results}\label{sec1}

One of the fundamental concepts in the study of \emph{structures} (i.e., sets endowed with additional structure in the form of operations and relations) is that of an automorphism, the formalisation of the intuitive notion of a \enquote{symmetry}. A significant portion of research across various disciplines is concerned with studying \enquote{highly symmetrical} structures $X$, a condition usually expressed in terms of certain transitivity assumptions on natural actions of the automorphism group $\Aut(X)$. For example, the well-studied notions of vertex-transitive graphs \cite[Definition 4.2.2, p.~85]{BW79a}, block- or flag-transitive designs \cite{CP93a,CP93b,Hub09a} and finite flag-transitive projective planes \cite{Tha03a} fall into this general framework.

In the special case where $X$ is a group $G$, the situation is more complicated, as the most straightforward transitivity assumption, that $\Aut(G)$ shall act transitively on $G$, is only satisfied by the trivial group. Hence, in order to obtain an interesting theory, weaker conditions have been studied by various authors. As examples, we mention
\begin{itemize}
\item the papers \cite{BD16a,DGB17a,LM86a,Stro02a} by various authors on finite groups $G$ such that $\Aut(G)$ has \enquote{few} orbits on $G$ (and the recent paper \cite{BDM20a} studying such a condition for infinite groups),
\item the first author's paper \cite{Bor19a} dealing with finite groups $G$ such that $\Aut(G)$ has at least one \enquote{large} orbit on $G$, and
\item Zhang's paper \cite{Zha92a} investigating finite groups $G$ where all elements of the same order are conjugate under $\Aut(G)$ (which were later discovered to have a connection with the celebrated CI-problem from algebraic graph theory, see \cite[Introduction]{LP97a}).
\end{itemize}
In his recent paper \cite{Bor20a}, the first author studied finite groups that are at the opposite end of the spectrum, i.e., they are \enquote{highly unsymmetrical}. Formally, the paper \cite{Bor20a} is concerned with finite groups $G$ such that all orbits of $\Aut(G)$ on $G$ have size bounded by a constant. It should be noted that Robinson and Wiegold already studied such conditions for general groups in their 1984 paper \cite{RW84a}, obtaining \emph{inter alia} a nice general structural characterisation \cite[Theorem 1]{RW84a} in the spirit of Neumann's celebrated characterisation of BFC-groups \cite[Theorem 3.1]{Neu54a}. However, the methods and results from \cite{Bor20a} are tailored to the finite case, where more specific statements can be made.

There are many instances where the full automorphism group $\Aut(G)$ is not \enquote{accessible}.  One notable case is where $G\leqslant\Sym(\Omega)$ is a permutation group. Here it is natural to only view those automorphisms of $G$ that arise from conjugation by an element from the normaliser of $G$ in $\Sym(\Omega)$. Formally, let us denote by $\Aut_{\perm}(G)$ the image of the conjugation action $\N_{\Sym(\Omega)}(G)\rightarrow\Aut(G)$.  It is then natural to ask which results about $\Aut(G)$ can be extended to $\Aut_{\perm}(G)$. This group of automorphisms has been previously considered as a means for computing the normaliser of $G$ in $\Sym(\Omega)$ \cite{Hul08}.

The goal of this paper is to extend the main results of \cite{Bor20a} to $\Aut_{\perm}(G)$. By a classical idea, dating back to Cayley, abstract groups are in a bijective correspondence with regular permutation groups via their (right) multiplication actions on themselves. Moreover, if $G$ is a regular permutation group, then $\Aut_{\perm}(G)=\Aut(G)$ (see Lemma \ref{regularAutLem} below), and so the statements of \cite[Theorem 1.1]{Bor20a} may equivalently be viewed as results on finite regular permutation groups $G$ such that all $\Aut_{\perm}(G)$-orbits on $G$ have size bounded by a constant. In this paper, we will extend \cite[Theorem 1.1]{Bor20a} from finite regular to finite transitive permutation groups. Our main results are as follows:

\begin{theorem}\label{mainTheo}
Let $G\leqslant\Sym(\Omega)$ be a finite transitive permutation group.
\begin{enumerate}
\item The following are equivalent:
\begin{enumerate}
\item All $\Aut_{\perm}(G)$-orbits on $G$ are of length at most $3$.
\item Up to isomorphism of permutation groups, $G$ is one of the following:
\begin{itemize}
\item $\IZ/m\IZ$ for some $m\in\{1,2,3,4,6\}$, $(\IZ/2\IZ)^2$ or $\Sym(3)$, each in its regular action on itself.
\item $\D_{2n}\leqslant\Sym(n)$, the symmetry group of a regular $n$-gon, for some $n\in\{3,4,6\}$.
\end{itemize}
\end{enumerate}
\item The order of $G$ cannot be bounded under the assumption that the maximum $\Aut_{\perm}(G)$-orbit length on $G$ is $4$.
\item Let $c$ be a positive and $d$ a non-negative integer. Assume that $G$ is $d$-generated and that all $\Aut_{\perm}(G)$-orbits on $G$ are of length at most $c$. Then $|G|\leqslant\ffrak(d,c^d)$, where $\ffrak$ is as in Notation \ref{mainNot} below.
\item If all $\Aut_{\perm}(G)$-orbits on $G$ are of length at most $23$, then $G$ is soluble.
\end{enumerate}
\end{theorem}

\begin{notation}\label{mainNot}
We denote by $\ffrak$ the function $\IN\times\IN^+\rightarrow\IN^+$ mapping
\[
(d,n)\mapsto 16^{(n+1)d}\cdot n^{2n^3d^2(5+d+4n^3\log_2{n})+2n^3+4nd+4d}.
\]
\end{notation}

The combination of statements (1) and (2) of Theorem \ref{mainTheo} is particularly interesting when compared to their counterparts in \cite[Theorem 1.1]{Bor20a}: At the moment, for finite abstract groups $G$, it is unknown what is the precise value of the maximum integer $c$ such that there are only finitely many $G$ with all $\Aut(G)$-orbits of length at most $c$ (we only know that $c\in\{3,4,5,6,7\}$). We also note that 23 in part (4) is sharp as $G=\Alt(5)$ in its usual action on five points has an $\Aut_{\perm}(G)$-orbit of length 24. 

We note that the proof of Theorem \ref{mainTheo}(3) uses a different, much simpler main idea than the one of \cite[Theorem 1.1(3)]{Bor20a}: Under the assumptions of Theorem \ref{mainTheo}(3), one has that $|\Aut_{\perm}(G)|\leqslant c^d$, which implies the asserted upper bound on $|G|$ by a partial generalisation of a classical theorem of Ledermann and Neumann, \cite[Theorem 6.6]{LN56a}, from abstract to transitive permutation groups, see Lemma \ref{ledNeuLem}. The same idea works in the case of abstract groups, where a direct application of \cite[Theorem 6.6]{LN56a} yields the following simpler and stronger bound compared to \cite[Theorem 1.1(3)]{Bor20a}:

\begin{theorem}\label{improvedBoundTheo}
Let $G$ be a finite abstract group. Assume that $G$ is $d$-generated and that all $\Aut(G)$-orbits on $G$ are of length at most $c$. Then
\[
|G|\leqslant c^{d\cdot c^d\cdot(1+\lfloor d\log_2{c}\rfloor)}+1.
\]
\end{theorem}

Unfortunately, as we will explain in Section \ref{sec5}, it seems that Ledermann and Neumann's proof of \cite[Theorem 6.6]{LN56a} cannot be adapted to transitive permutation groups, requiring us to use Sambale's recent proof from \cite{Sam19a} instead. Sambale's proof is conceptually simpler than Ledermann and Neumann's argument, but it produces worse explicit bounds, which explains the stark contrast between the bounds in Theorems \ref{mainTheo}(3) and \ref{improvedBoundTheo}.

\section{Preliminaries}\label{sec2}

\subsection{Notation and terminology}\label{subsec2P1}

We denote by $\IN$ the set of natural numbers (including $0$) and by $\IN^+$ the set of positive integers. The symbol $\phi$ denotes Euler's totient function.

As in \cite{Bor20a}, for a finite abstract group $G$, we denote by $\maol(G)$ the maximum length of an orbit of the natural action of the automorphism group $\Aut(G)$ on $G$. On the other hand, if $G\leqslant\Sym(\Omega)$ is a permutation group, then as in Section \ref{sec1}, we denote by $\Aut_{\perm}(G)$ the group of all automorphisms of $G$ that are induced through conjugation by some element from the normaliser $\N_{\Sym(\Omega)}(G)$, and if $G$ is finite, we denote by $\maol_{\perm}(G)$ the maximum length of an orbit of the natural action of $\Aut_{\perm}(G)$ on $G$. Note that we may also view $G$ as an abstract group, so that the notation $\Aut(G)$ is well-defined (as is $\maol(G)$ if $G$ is finite), and we have $\Aut_{\perm}(G)\leqslant\Aut(G)$, as well as $\maol_{\perm}(G)\leqslant\maol(G)$ if $G$ is finite.

If $G$ is an abstract group, we denote by $G_{\reg}\leqslant\Sym(G)$ the image of the \mbox{(right-)}regular permutation representation of $G$ on itself. The minimum number of generators of a finitely generated group $G$ will be denoted by $d(G)$ and called the \emph{rank of $G$}. The exponent (least common multiple of the element orders) of a finite group $G$ will be denoted by $\Exp(G)$. The soluble radical (largest soluble normal subgroup) of a finite group $G$ will be denoted by $\Rad(G)$, and the centre of $G$ by $\zeta G$.

\subsection{Some basic results}\label{subsec2P2}

In this subsection, we collect some auxiliary results (most if not all of which are well-known) that will be used in the proofs of the main results. The following three lemmas deal with transitive permutation groups. We include the first without proof.

\begin{lemmma}\label{transAbLem}
Let $G$ be a transitive permutation group, and let $S$ be a point stabiliser in $G$. The following hold:
\begin{enumerate}
\item $S$ is \emph{core-free in $G$}, i.e., $S$ does not contain any nontrivial normal subgroup of $G$.
\item If $G$ is abelian, then $G$ is regular.
\end{enumerate}
\end{lemmma}

\begin{lemmma}\label{autPermLem}
Let $G$ be a transitive permutation group, and let $\alpha\in\Aut(G)$. The following are equivalent:
\begin{enumerate}
\item $\alpha\in\Aut_{\perm}(G)$.
\item For some (equivalently, any) point stabiliser $S$ in $G$, we have that $S^{\alpha}$ is also a point stabiliser in $G$.
\end{enumerate}
\end{lemmma}

\begin{proof}
See \cite[Theorem 4.2B, p.~110]{DM96a}.
\end{proof}

\begin{lemmma}\label{regularAutLem}
Let $G$ be a regular permutation group. Then $\Aut_{\perm}(G)=\Aut(G)$. In particular, if $G$ is finite, then $\maol_{\perm}(G)=\maol(G)$.
\end{lemmma}

\begin{proof}
This is immediate by Lemma \ref{autPermLem}.
\end{proof}

The next two lemmas are concerned with finite abelian groups:

\begin{lemmma}\label{centAbelianLem}
Let $A$ be a finite abelian group, and let $B$ be a proper subgroup of $A$. The following are equivalent:
\begin{enumerate}
\item The centraliser of $B$ in $\Aut(A)$ is trivial.
\item $|B|$ is odd, and $|A|=2|B|$ (or, equivalently, $A=B\times\IZ/2\IZ$).
\end{enumerate}
\end{lemmma}

\begin{proof}
The implication \enquote{(2)$\Rightarrow$(1)} is clear by the fact that $\Aut(\IZ/2\IZ)$ is trivial, so we focus on the proof of \enquote{(1)$\Rightarrow$(2)}, for which we will first show the following claim: \enquote{If $G$ is a finite abelian $p$-group for some prime $p$, and $H$ is a proper subgroup of $G$, then $\C_{\Aut(G)}(H)$ is trivial if and only if $|G|=2$ and $|H|=1$.}

In order to prove the claim, note first that it is clear if $H$ is trivial, as $\IZ/2\IZ$ is the only nontrivial (finite) group with trivial automorphism group. We may thus assume that $H$ is nontrivial, and under this assumption, we need to show that $\C_{\Aut(G)}(H)\not=\{\id_G\}$. For this, we may and will assume w.l.o.g.~that $H$ is of index $p$ in $G$. Fix elements $g\in G\setminus H$ and $h_0\in H$ with $\ord(h_0)=p$. Then every element of $G$ has a unique representation as $ig+h$ for some $i\in\{0,1,\ldots,p-1\}$ and some $h\in H$. We define a function $\alpha:G\rightarrow G$ via $(ig+h)^{\alpha}:=ig+ih_0+h$. It is easy to check that this function is a nontrivial element of $\C_{\Aut(G)}(H)$, which concludes the proof of the claim.

Now that the claim has been proved, we will show the contraposition of \enquote{(1)$\Rightarrow$(2)}: Assume that $|B|$ is even or $|A|>2|B|$; we need to show that $\C_{\Aut(A)}(B)$ is nontrivial. For each prime $p$, denote by $A_p$ and $B_p$, the Sylow $p$-subgroups of $A$ and $B$ respectively. Note that by Sylow's Theorems, for all primes $p$ we have $B_p<A_p$. If $p>2$, then by the claim, $\C_{\Aut(A_p)}(B_p)$ is nontrivial, whence $\C_{\Aut(A)}(B)$ is nontrivial, as required. We may thus assume that $B_p=A_p$ for all $p>2$ and $B_2<A_2$. Then under either of the assumptions \enquote{$|B|$ is even} or \enquote{$|A|>2|B|$}, we find that $(|A_2|,|B_2|)\not=(2,1)$, whence $\C_{\Aut(A_2)}(B_2)$ is nontrivial by the claim, and thus $\C_{\Aut(A)}(B)$ is nontrivial, as required.
\end{proof}

For the second lemma on finite abelian groups, we require a simple definition:

\begin{deffinition}\label{abelianBasisDef}
Let $p$ be a prime, and let $A=\IZ/p^{e_1}\IZ\times\cdots\times\IZ/p^{e_r}\IZ$ with $e_1\geqslant e_2\geqslant\cdots\geqslant e_r\geqslant1$ be a finite abelian $p$-group of rank $r$. A \emph{basis of $A$} is an $r$-tuple $(a_1,\ldots,a_r)\in A^r$ such that $\ord(a_i)=p^{e_i}$ for $i=1,\ldots,r$ and $A=\langle a_1\rangle\times\cdots\times\langle a_r\rangle$.
\end{deffinition}

\begin{lemmma}\label{elAbBasisLem}
Let $p$ be a prime, and let $A=\IZ/p^{e_1}\IZ\times\cdots\times\IZ/p^{e_r}\IZ$ with $e_1\geqslant e_2\geqslant\cdots\geqslant e_r\geqslant1$ be a finite abelian $p$-group of rank $r$. Moreover, let $B$ be an elementary abelian subgroup of $A$ of rank $s$. Then there is a basis $(a_1,\ldots,a_r)$ of $A$ as well as indices $1\leqslant i_1<i_2<\cdots<i_s\leqslant r$ such that $(p^{e_{i_1}-1}a_{i_1},p^{e_{i_2}-1}a_{i_2},\ldots,p^{e_{i_s}-1}a_{i_s})$ is a basis of $B$.
\end{lemmma}

\begin{proof}
We proceed by induction on $s$. The induction base, $s=0$, is vacuously true. Assume thus that $s\geqslant1$ and that the assertion is true for elementary abelian subgroups of rank at most $s-1$. Fix a nontrivial element $b\in B$. By \cite[Lemma 5]{Sam19a}, there are subgroups $A_1,A_2\leqslant A$ such that $b\in A_1$, $A_1$ is cyclic, and $A=A_1\times A_2$. Fix a direct complement $B_2$ of $B_1:=\langle b\rangle$ in $B$. Through subtracting suitable multiples of $a$ from the entries of any given basis of $B_2$, we may assume without loss of generality that $B_2\leqslant A_2$. The assertion now follows through applying the induction hypothesis to $B_2$ and $A_2$.
\end{proof}

\begin{remmark}
Consider the following assertion, which generalises the statement of Lemma \ref{elAbBasisLem}:

\enquote{For every prime $p$, every finite abelian $p$-group $A$ and every subgroup $B$ of $A$, there are bases $\vec{b}$ and $\vec{a}$ of $B$ and $A$ respectively such that for each entry $b$ of $\vec{b}$, some entry of $\vec{a}$ is a root of $b$.}

This assertion is \emph{not} true. For example, for an arbitrary prime $p$, consider $A=\IZ/p^3\IZ\times\IZ/p\IZ$, with basis $(a_1,a_2)$, and set $b:=pa_1+a_2$ and $B:=\langle b\rangle\cong\IZ/p^2\IZ$. Then if the above assertion was true, the generator $b$ of $B$ would need to have a $p$-th root in $A$, which it does not.
\end{remmark}

Finally, as in \cite{Bor20a}, we will be using the concept of a \enquote{central automorphism} at several points in our arguments, and we briefly state the most important facts concerning this concept (which are well-known and easy to check). For each group $G$ and each group homomorphism $f:G\rightarrow\zeta G$, the function $\alpha_f:G\rightarrow G$, $g\mapsto g\cdot f(g)$, is an endomorphism of $G$ called the \emph{central endomorphism of $G$ associated with $f$}. The kernel of $\alpha_f$ consists of those $g\in\zeta G$ such that $f(g)=g^{-1}$. In particular, if $G$ is finite, then $\alpha_f$ is an automorphism of $G$ if and only if the only element of $\zeta G$ that is inverted by $f$ is $1_G$. Such automorphisms are called \emph{central automorphisms}, and they form a subgroup of $\Aut(G)$ denoted by $\Aut_{\cent}(G)$.

\section{Proof of Theorem \ref{mainTheo}(1)}\label{sec3}

We split the proof of Theorem \ref{mainTheo}(1) into the three cases $\maol_{\perm}(G)=1,2,3$, each dealt with in one of the following three subsections.

\subsection{Finite transitive permutation groups with maximum normaliser orbit length 1}\label{subsec3P1}

The following is a simple consequence of known results:

\begin{propposition}\label{maol1Prop}
Let $G$ be a finite transitive permutation group. The following are equivalent:
\begin{enumerate}
\item $\maol_{\perm}(G)=1$.
\item Up to permutation group isomorphism, $G$ is one of $(\IZ/m\IZ)_{\reg}$ for $m\in\{1,2\}$.
\end{enumerate}
\end{propposition}

\begin{proof}
The implication \enquote{(2)$\Rightarrow$(1)} is easy, so we focus on the implication \enquote{(1)$\Rightarrow$(2)}. Since $\Inn(G)\leqslant\Aut_{\perm}(G)$, all conjugacy classes of $G$ are of length $1$, i.e., $G$ is abelian. Hence, by Lemma \ref{transAbLem}(2), $G$ is regular, and by Lemma \ref{regularAutLem}, $\maol(G)=1$. The result now follows from \cite[Proposition 3.1.1]{Bor20a}.
\end{proof}

\subsection{Finite transitive permutation groups with maximum normaliser orbit length 2}\label{subsec3P2}

In this subsection, we will prove the following result:

\begin{propposition}\label{maol2Prop}
Let $G\leqslant\Sym(\Omega)$ be a finite transitive permutation group. The following are equivalent:
\begin{enumerate}
\item $\maol_{\perm}(G)=2$.
\item Up to permutation group isomorphism, $G$ is one of the following:
\begin{itemize}
\item $(\IZ/m\IZ)_{\reg}$ for some $m\in\{3,4,6\}$, or
\item $\D_8\leqslant\Sym(4)$.
\end{itemize}
\end{enumerate}
\end{propposition}

\begin{proof}
The implication \enquote{(2)$\Rightarrow$(1)} is easy, so we will be concerned with the implication \enquote{(1)$\Rightarrow$(2)}. If $G$ is regular, then by Lemma \ref{regularAutLem} and \cite[Proposition 3.2.4]{Bor20a}, it follows that $G\cong(\IZ/m\IZ)_{\reg}$ for some $m\in\{3,4,6\}$. We will thus henceforth assume that $G$ is nonregular (hence nonabelian by Lemma \ref{transAbLem}(2)), and under this assumption, we will show that $G\cong\D_8\leqslant\Sym(4)$.

By assumption, all conjugacy classes of $G$ are of length at most $2$, and so the exponent of $\Inn(G)$ is $2$, whence $\Inn(G)$ is an elementary abelian $2$-group. In particular, $G$ is nilpotent of class $2$, and all Sylow $p$-subgroups of $G$ for $p>2$ are abelian. It follows that $G=G_2\times A$ where $G_2$ is a nonabelian $2$-group of class $2$ and $A$ is a finite abelian group of odd order. Fix a point $\omega\in\Omega$, and consider the associated point stabiliser $G_{\omega}\leqslant G$. Since $G_{\omega}$ is core-free in $G$, we find that $G_{\omega}\cap A=\{1_G\}$, or equivalently (using the coprimality of $|G_2|$ and $|A|$), $G_{\omega}\leqslant G_2$.

We claim that $A$ is trivial. Assume otherwise. Then by \cite[Proposition 3.1.1]{Bor20a}, there is an $a\in A$ such that $|a^{\Aut(A)}|\geqslant2$. Let $g\in G_2\setminus\zeta G_2$, and consider the element $h:=ga\in G$. Note that $\Aut(G)=\Aut(G_2)\times\Aut(A)$, and since $G_{\omega}\leqslant G_2$, it follows by Lemma \ref{autPermLem} that $\Inn(G_2)\times\Aut(A)\leqslant\Aut_{\perm}(G)$. Hence
\[
|h^{\Aut_{\perm}(G)}|\geqslant|h^{\Inn(G_2)\times\Aut(A)}|=|g^{G_2}|\cdot|a^{\Aut(A)}|\geqslant 2\cdot2=4>2,
\]
a contradiction. Therefore, $A$ is trivial, and so $G=G_2$ is a nonabelian $2$-group of class $2$.

Observing once more that $G_{\omega}$ is core-free in $G$, we find that $G_{\omega}\cap\zeta G=\{1_G\}$. As $G/\zeta G=\Inn(G)$ is (as noted above) an elementary abelian $2$-group, it follows that $G_{\omega}$ is also an elementary abelian $2$-group, embedded into $G/\zeta G$ via the canonical projection $G\rightarrow G/\zeta G$. Note that $|G_{\omega}|<|G/\zeta G|$, since otherwise, $G=G_{\omega}\times\zeta G$, which is impossible, since $G_{\omega}$ is both nontrivial and core-free in $G$.

We claim that $\zeta G$ is cyclic. Assume otherwise. Then, fixing an embedding $(\IZ/2\IZ)^2\overset{\iota}\hookrightarrow\zeta G$ and a projection $(G/\zeta G)/(G_{\omega}\zeta G/\zeta G)\overset{\pi}\twoheadrightarrow\IZ/2\IZ$, we have four distinct group homomorphisms
\[
f:G\overset{\text{can.}}\twoheadrightarrow G/\zeta G\overset{\text{can.}}\twoheadrightarrow(G/\zeta G)/(G_{\omega}\zeta G/\zeta G)\overset{\pi}\twoheadrightarrow\IZ/2\IZ\hookrightarrow(\IZ/2\IZ)^2\overset{\iota}\hookrightarrow\zeta G.
\]
These homomorphisms $f$ satisfy  $\zeta G\leqslant\ker(f)$,  the associated central automorphisms $\alpha_f$ centralise $G_{\omega}$ (in particular, $\alpha_f\in\Aut_{\perm}(G)$ by Lemma \ref{autPermLem}) and  a suitable element of $G$ has four distinct images under those automorphisms $\alpha_f$. It follows that $\maol_{\perm}(G)\geqslant 4>2$, a contradiction. This concludes the proof that $\zeta G$ is cyclic.

Note that $G'\leqslant\zeta G$ is also cyclic. But since $G/\zeta G$ has exponent $2$, it follows that for all $x,y\in G$,
\[
1_G=[x^2,y]=[x,y]^x[x,y]=[x,y]^2,
\]
whence $G'\cong\IZ/2\IZ$. Next, we claim that $G/G'$ is an elementary abelian $2$-group. Assume otherwise. Then $|\zeta G|\geqslant 4$ (otherwise, $G$ is extraspecial). Denote the image of $G_{\omega}$ under the canonical projection $G\rightarrow G/G'$ by $P$. Since $P$ is elementary abelian, we may apply Lemma \ref{elAbBasisLem} to find a basis $\vec{b}=(b_1,\ldots,b_n)$ of $G/G'$ such that a suitable subsequence of $\vec{b}$ \enquote{powers up} to a basis of $P$. Now, the first basis entry $b_1$ has order $2^k$ for some $k\geqslant2$, and using the facts that $G'\leqslant\zeta G$, that $G_{\omega}\cap\zeta G=\{1_G\}$ and that every square in $G$ lies in $\zeta G$, we conclude that $b_1^{2^{k-1}}\notin P$. It follows that $P\leqslant\langle b_2,\ldots,b_n\rangle$, whence we may fix a projection $\pi_1:G/(G_{\omega}G')\twoheadrightarrow\IZ/2^k\IZ$. We also fix a projection $\pi_2:\IZ/2^k\IZ\twoheadrightarrow\IZ/4\IZ$. Through composition, we obtain four distinct group homomorphisms
\[
f:G\overset{\text{can.}}\twoheadrightarrow G/(G_{\omega}G')\overset{\pi_1}\twoheadrightarrow\IZ/2^k\IZ\overset{\pi_2}\twoheadrightarrow\IZ/4\IZ\rightarrow\zeta G,
\]
each of which has the property that $1_G$ is the only element of $\zeta G$ that is inverted by $f$. Hence we have four distinct associated central automorphisms $\alpha_f\in\Aut_{\perm}(G)$, and a suitable element of $G$ has four distinct images under these automorphisms, whence $\maol_{\perm}(G)\geqslant 4>2$, a contradiction. This concludes the proof that $G/G'$ is elementary abelian.

Now, setting $d:=\log_2{|G/G'|}$ and $t:=\log_2{|G_{\omega}|}$, we define a \emph{standard tuple in $G$} as a tuple $(g_1,\ldots,g_d)\in G^d$ such that
\begin{itemize}
\item $(g_1,\ldots,g_t)$ is a basis of $G_{\omega}$, and
\item the entry-wise image of $(g_1,\ldots,g_d)$ under the canonical projection $G\rightarrow G/G'$ is a basis of $G/G'$.
\end{itemize}
If $\vec{g}=(g_1,\ldots,g_d)\in G^d$ is a standard tuple in $G$, then the \emph{power-commutator tuple associated with $\vec{g}$} is the $(d+{d \choose 2})$-tuple
\[
(g_1^2,g_2^2,\ldots,g_d^2,[g_1,g_2],[g_1,g_3],\ldots,[g_1,g_d],[g_2,g_3],[g_2,g_4],\ldots,[g_2,g_d],\ldots,[g_{d-1},g_d])
\]
with entries in $G'$. Two standard tuples in $G$ are called \emph{equivalent} if and only if they have the same power-commutator tuple.

As in \cite[proof of Proposition 3.2.4]{Bor20a}, two standard tuples in $G$ are conjugate under the component-wise action of $\Aut(G)$ if and only if they are equivalent. However, the above definition of \enquote{standard tuple} differs from the one in \cite[proof of Proposition 3.2.4]{Bor20a}, and it was chosen in such a way that any automorphism $\alpha$ of $G$ which maps any given standard tuple in $G$ to any other given standard tuple in $G$ has the property that $G_{\omega}^{\alpha}=G_{\omega}$, whence $\alpha\in\Aut_{\perm}(G)$ by Lemma \ref{autPermLem}. It follows that each equivalence class of standard tuples in $G$ is contained in an $\Aut_{\perm}(G)$-orbit on $G^d$. Considering that the number of equivalence classes of standard tuples in $G$ is at most $2^{d+{d\choose 2}}$ (this uses that $|G'|=2$) and that the number of standard tuples in $G$ is at least
\[
\prod_{i=0}^{t-1}{(2^t-2^i)}\cdot\prod_{j=t}^{d-1}{(2^d-2^j)}\cdot 2^{d-t}=2^{d+{d\choose 2}-t}\cdot\prod_{i=1}^t{(2^i-1)}\cdot\prod_{j=1}^{d-1}{(2^j-1)},
\]
it follows that there is an equivalence class of standard tuples in $G$ (and thus an $\Aut_{\perm}(G)$-orbit on $G^d$) of size at least
\begin{align*}
&\frac{2^{d+{d\choose 2}-t}\cdot\prod_{i=1}^t{(2^i-1)}\cdot\prod_{j=1}^{d-t}{(2^j-1)}}{2^{d+{d\choose 2}}}=\frac{\prod_{i=1}^t{(2^i-1)}\cdot\prod_{j=1}^{d-t}{(2^j-1)}}{2^t}= \\
&(1-\frac{1}{2^t})\cdot 2^{{t\choose 2}}\cdot\prod_{i=1}^{t-1}{(1-\frac{1}{2^i})}\cdot 2^{{{d-t+1}\choose 2}}\cdot\prod_{j=1}^{d-t}{(1-\frac{1}{2^j})}\geqslant\frac{1}{2}\cdot 2^{{t\choose 2}+{{d-t+1}\choose 2}}\cdot(\prod_{i=1}^{\infty}{(1-\frac{1}{2^i})})^2 \\
&\geqslant\frac{1}{2}\cdot 2^{\frac{d}{4}\cdot(\frac{d}{2}-1)}\cdot 0.28^2=0.0392\cdot 2^{\frac{d^2}{8}-\frac{d}{4}},
\end{align*}
where the first inequality uses that $t\geqslant1$ (since $G$ is nonregular), and the second inequality uses that
\[
\min\{{t\choose 2},{{d-t+1}\choose 2}\}\geqslant\frac{d}{4}\cdot(\frac{d}{2}-1).
\]
However, since all $\Aut_{\perm}(G)$-orbits on $G$ are of length at most $2$, it follows that the length of an $\Aut_{\perm}(G)$-orbit on $G^d$ cannot exceed $2^d$. Therefore,
\[
2^d \geqslant 0.0392\cdot 2^{\frac{d^2}{8}-\frac{d}{4}},
\]
which implies that $d\leqslant12$. Moreover, for $d=2,3,\ldots,12$ and $t\in\{1,2,\ldots,d-1\}$, one can check that
\[
(1-\frac{1}{2^t})\cdot 2^{{t\choose 2}}\cdot\prod_{i=1}^{t-1}{(1-\frac{1}{2^i})}\cdot 2^{{{d-t+1}\choose 2}}\cdot\prod_{j=1}^{d-t}{(1-\frac{1}{2^j})} > 2^d
\]
unless
\[
(d,t)\in\{(2,1),(3,1),(3,2),(4,1),(4,2),(4,3),(5,2),(5,3),(5,4),(6,3),(6,4)\},
\]
so these are the only possibilities for $(d,t)$. In particular, the degree of $G$, which equals $2^{d-t+1}$, is at most $16$. To conclude this proof, we will make use of the library of finite transitive permutation groups of small degree in GAP \cite{GAP4}, which was implemented by Hulpke \cite{TransGrp}, goes up to degree $32$ and is based on the papers \cite{But93a,BK83a,CH08a,Hul05a,Roy87a} as well as an unpublished classification, due to Sims, of the primitive permutation groups up to degree $50$ (to cover the transitive groups of degree $31$); see also \cite{HR19a} for the latest record concerning the classification of transitive groups of small degree, which goes up to degree $48$. In any case, this classification allows us to conclude that the only transitive permutation group $G$ of degree at most $16$ such that
\begin{itemize}
\item $G$ is a nonabelian $2$-group of class $2$,
\item $\zeta G$ is cyclic,
\item $G/G'$ is elementary abelian, and
\item $\maol_{\perm}(G)=2$
\end{itemize}
is $\D_8\leqslant\Sym(4)$.
\end{proof}

\subsection{Finite transitive permutation groups with maximum normaliser orbit length 3}\label{subsec3P3}

In this subsection, we will be concerned with the proof of the following part of Theorem \ref{mainTheo}(1):

\begin{propposition}\label{maol3Prop}
Let $G\leqslant\Sym(\Omega)$ be a finite transitive permutation group. The following are equivalent:
\begin{enumerate}
\item $\maol_{\perm}(G)=3$.
\item Up to permutation group isomorphism, $G$ is one of the following:
\begin{itemize}
\item $((\IZ/2\IZ)^2)_{\reg}$, $\Sym(3)_{\reg}$, or
\item $\D_{2n}\leqslant\Sym(n)$ for some $n\in\{3,6\}$.
\end{itemize}
\end{enumerate}
\end{propposition}

We will first show two auxiliary results, which are extensions of \cite[Lemmas 3.3.1(2) and 3.3.2]{Bor20a}.

\begin{lemmma}\label{auxLem1}
Let $G\leqslant\Sym(\Omega)$ be a finite, transitive, nonregular permutation group such that $\maol_{\perm}(G)=3$. Then the following hold:
\begin{enumerate}
\item The set of element orders of $\Aut_{\perm}(G)$ is contained in $\{1,2,3\}$.
\item $G$ is a $\{2,3\}$-group.
\end{enumerate}
\end{lemmma}

\begin{proof}
For (1): As in \cite[proof of Lemma 3.3.1(2,a)]{Bor20a}, for all $\alpha\in\Aut_{\perm}(G)$, one has that $G=\C_G(\alpha^2)\cup\C_G(\alpha^3)$, and thus $G=\C_G(\alpha^2)$ or $G=\C_G(\alpha^3)$, as $G$ is not the union of two proper subgroups.

For (2): As in \cite[proof of Lemma 3.3.1(2,b)]{Bor20a}, one can show that $G=G_{\{2,3\}}\times G_{\{2,3\}'}$, where $G_{\{2,3\}}$ is the unique Hall-$\{2,3\}$-subgroup of $G$, and $G_{\{2,3\}'}$ is the unique, central Hall-$\{2,3\}'$-subgroup of $G$. Fix a point $\omega\in\Omega$, and consider the point stabiliser $G_{\omega}\leqslant G$. Since $G_{\omega}$ is core-free in $G$, it follows that $G_{\omega}\cap G_{\{2,3\}'}=\{1_G\}$, and thus $G_{\omega}\leqslant G_{\{2,3\}}$, whence $\Inn(G_{\{2,3\}})\times\Aut(G_{\{2,3\}'})$ embeds naturally into $\Aut_{\perm}(G)$ by Lemma \ref{autPermLem}. Hence, if $|G_{\{2,3\}'}|>1$, then it follows by \cite[Lemma 3.1(2)]{Bor20a} that $\maol_{\perm}(G)\geqslant 4>3$, a contradiction.
\end{proof}

\begin{lemmma}\label{auxLem2}
Let $G\leqslant\Sym(\Omega)$ be a finite, transitive, nonregular permutation group such that $\maol_{\perm}(G)=3$. Then the set of element orders of $\Inn(G)$ is exactly $\{1,2,3\}$.
\end{lemmma}

\begin{proof}
Assume otherwise. By Lemma \ref{transAbLem}(2), $G$ is nonabelian, and so by Lemma \ref{auxLem1}, $\Exp(\Inn(G))\in\{2,3\}$; in any case, $\Inn(G)$, and thus $G$ itself, is nilpotent. By Lemma \ref{auxLem1}(2), we can write $G=G_2\times G_3$ where $G_p$ denotes the unique Sylow $p$-subgroup of $G$ for $p\in\{2,3\}$. Fix a point $\omega\in\Omega$ and consider the point stabiliser $G_{\omega}\leqslant G$. We make a case distinction:
\begin{enumerate}
\item Case: $\Exp(\Inn(G))=2$. Then $G_2$ is nonabelian and $G_3$ is abelian (and thus central in $G$). Since $G_{\omega}$ is core-free in $G$, it follows that $G_{\omega}\cap G_3=\{1_G\}$, and thus $G_{\omega}\leqslant G_2$. Therefore, $\Inn(G_2)\times\Aut(G_3)$ embeds naturally into $\Aut_{\perm}(G)$, and so if $|G_3|>1$, we could conclude that $\maol_{\perm}(G)\geqslant 2\cdot 2=4>3$, a contradiction. Consequently, $G_3$ is trivial, and so $G=G_2$ is a nonabelian $2$-group with $\maol_{\perm}(G)=3$. An argument analogous to the one in \cite[proof of Lemma 3.3.2, Case (1)]{Bor20a} yields the final contradiction for this case.
\item Case: $\Exp(\Inn(G))=3$. Then $G_2$ is abelian (hence central in $G$), whereas $G_3$ is nonabelian. We have $G_{\omega}\leqslant G_3$, and so, viewing $H:=G_3$ as a permutation group via the inclusions $H\leqslant G\leqslant\Sym(\Omega)$, we find that $H_{\omega}=G_{\omega}$ and $\maol_{\perm}(H)=3$. As in \cite[proof of Lemma 3.3.2, Case (2)]{Bor20a}, we conclude that $\Inn(H)$ is abelian, i.e., that the nilpotency class of $H$ is $2$. Moreover, in view of $\Exp(\Inn(H))=3$, we conclude that $\Inn(H)$ is an elementary abelian $3$-group. One can now show the following facts, in the listed order and analogously to the proof of Proposition \ref{maol2Prop}:
\begin{itemize}
\item $H_{\omega}$ is an elementary abelian $3$-group, embedded into $H/\zeta H\cong\Inn(H)$ via the canonical projection $H\rightarrow H/\zeta H$.
\item $|H_{\omega}|<|H/\zeta H|$.
\item $\zeta H$ is cyclic.
\item $|H'|=3$.
\item $H/H'$ is an elementary abelian $3$-group.
\end{itemize}
We can use these restrictions on $H$ to carry out an analogue of the \enquote{standard tuples} argument from the proof of Proposition \ref{maol2Prop} (only needing to replace the prime $2$ by $3$), which allows us to conclude that with $d:=\log_3{|H/H'|}$ and $t:=\log_3{|H_{\omega}|}$, we have
\begin{equation}\label{3tupleEq}
(1-\frac{1}{3^t})\cdot 3^{{t \choose 2}+{{d-t+1} \choose 2}}\cdot\prod_{i=1}^{t-1}{(1-\frac{1}{3^i})}\cdot\prod_{j=1}^{d-t}{(1-\frac{1}{3^j})} \leqslant 3^d,
\end{equation}
in particular
\[
3^d \geqslant \frac{2}{3}\cdot 3^{\frac{d}{4}(\frac{d}{2}-1)}\cdot(\prod_{i=1}^{\infty}{(1-\frac{1}{3^i})})^2 \geqslant \frac{2}{3}\cdot 3^{\frac{d^2}{8}-\frac{d}{4}}\cdot 0.56^2,
\]
which only holds for $d\leqslant11$. Moreover, among all pairs $(d,t)$ with $d\in\{2,3,\ldots,11\}$ and $t\in\{1,2,\ldots,d-1\}$, the stronger inequality from Formula (\ref{3tupleEq}) only holds for
\[
(d,t)\in\{(2,1),(3,1),(3,2),(4,2),(4,3)\},
\]
so that $\deg(H)=3^{d-t+1}\leqslant 3^3=27$. However, according to the library of finite transitive permutation groups of small degree \cite{GAP4,TransGrp}, there are no transitive permutation groups $H$ of degree at most $27$ such that
\begin{itemize}
\item $H$ is a nonabelian $3$-group of class $2$,
\item $\zeta H$ is cyclic,
\item $H/H'$ is elementary abelian, and
\item $\maol_{\perm}(H)=3$,
\end{itemize}
a contradiction to the case assumption.\qedhere
\end{enumerate}
\end{proof}

We are now ready to prove Proposition \ref{maol3Prop}.

\begin{proof}[Proof of Proposition \ref{maol3Prop}]
The implication \enquote{(2)$\Rightarrow$(1)} is easy, so we focus on the proof of \enquote{(1)$\Rightarrow$(2)}. The regular case is dealt with in \cite[Proposition 3.3.3]{Bor20a}, so we assume that $G$ is nonregular. By Lemma \ref{auxLem2}, the set of element orders of $\Inn(G)$ is $\{1,2,3\}$, and as in \cite[proof of Proposition 3.3.3]{Bor20a}, this allows us to conclude that $\Inn(G)\cong\Sym(3)$. Fix a point $\omega\in\Omega$, and consider the point stabiliser $G_{\omega}\leqslant G$. Since $\Inn(G)\cong G/\zeta G$ and $G_{\omega}\cap\zeta G=\{1_G\}$, we find that $G_{\omega}$ is embedded into $\Sym(3)$ via the canonical projection $G\rightarrow G/\zeta G$. We make a case distinction:
\begin{enumerate}
\item Case: $G_{\omega}\cong\Sym(3)$. Then $G=G_{\omega}\times\zeta G$, which is impossible, since $G_{\omega}$ is core-free in $G$.
\item Case: $G_{\omega}\cong\IZ/3\IZ$. We claim that $\zeta G$ is a $3$-group. Indeed, assuming that $2$ divides $|\zeta G|$, there is a group homomorphism chain of the form
\[
f:G\overset{\text{can.}}\twoheadrightarrow G/\zeta G\overset{\sim}\rightarrow\Sym(3)\twoheadrightarrow\IZ/2\IZ\hookrightarrow\zeta G.
\]
The corresponding central automorphism $\alpha_f$ of $G$ lies in $\Aut_{\perm}(G)$ and maps any fixed element $g\in G$ that projects onto an order $2$ element in $\Sym(3)$ to a different element in the same central coset. Therefore and since the image of $g$ in $G/\zeta G$ has conjugacy class length $3$, it follows that $\maol_{\perm}(G)\geqslant 2\cdot 3>3$, a contradiction. This concludes the proof that $\zeta G$ is a $3$-group, and as in \cite[proof of Proposition 3.3.3]{Bor20a}, we can infer from this that $G=\zeta G\times\Sym(3)$.

We next claim that $|\zeta G|\leqslant 3$. Assume otherwise. With respect to a fixed direct decomposition of $G$ of the form $\zeta G\times\Sym(3)$, denote by $P$ the projection of $G_{\omega}$ to $\zeta G$. Note that $|P|=3$, because if $P$ is trivial, then $G_{\omega}$ is normal and hence not core-free in $G$. Moreover, denoting by $\tau$ a fixed element of order $2$ in $\Sym(3)$, observe that $\Stab_{\Aut(\zeta G)}(P)$ embeds naturally into $\Aut_{\perm}(G)$, via the injective group homomorphism
\begin{align*}
&\alpha \mapsto \begin{cases}((z,\sigma) \mapsto (z^{\alpha},\sigma)), & \text{if }\alpha\in\C_{\Aut(\zeta G)}(P), \\ ((z,\sigma)\mapsto (z^{\alpha},\sigma^{\tau})), & \text{otherwise}\end{cases} \\
&\text{for all }\alpha\in\Stab_{\Aut(\zeta G)}(P),z\in\zeta G,\sigma\in\Sym(3).
\end{align*}
Therefore, $\maol_{\perm}(G)$ is at least the maximum orbit length of $\Stab_{\Aut(\zeta G)}(P)$ on $\zeta G$, which we will show to be strictly larger than $3$ in the following subcase distinction:
\begin{enumerate}
\item Subcase: $\zeta G$ is cyclic. Then $P$ is characteristic in $\zeta G$, so $\Stab_{\Aut(\zeta G)}(P)=\Aut(\zeta G)$. Therefore, writing $|\zeta G|=3^k$ (with $k\geqslant2$), we conclude that the maximum orbit length of $\Stab_{\Aut(\zeta G)}(P)$ on $\zeta G$ is $\phi(3^k)=3^{k-1}\cdot 2\geqslant 6>3$.
\item Subcase: $\zeta G$ is not cyclic. By Lemma \ref{elAbBasisLem}, there is a basis $(z_1,\ldots,z_r)$ of $\zeta G$ and an $i\in\{1,\ldots,r\}$ such that $P\leqslant\langle z_i\rangle$. Therefore, if $\ord(z_j)=3^k>3$ for some $j\in\{1,\ldots,r\}$, then $\Stab_{\Aut(\zeta G)}(P)$ has an orbit of length $\phi(3^k)=3^{k-1}\cdot 2\geqslant6>3$, a contradiction. Hence $\zeta G\cong(\IZ/3\IZ)^r$, so that the maximum orbit length of $\Stab_{\Aut_{\zeta G}}(P)$ on $\zeta G$ is $3^r-3\geqslant6>3$, another contradiction.
\end{enumerate}
This concludes the proof that $|\zeta G|\leqslant 3$. However, note that if $|\zeta G|=1$, then $G_{\omega}$ is normal and thus not core-free in $G$. It follows that $G$ is of order $18$ and of degree $6$, and by the library of finite transitive permutation groups of small degree \cite{GAP4,TransGrp}, there are no transitive permutation groups with this combination of order and degree and with maximum normaliser orbit length $3$, a contradiction.
\item Case: $G_{\omega}\cong\IZ/2\IZ$. In what follows, for a finite group $H$ and a prime $p$, if $H$ has a unique Sylow $p$-subgroup, we denote that subgroup by $H_p$. Let $A\leqslant G$ be the preimage of the unique index $2$ subgroup of $\Sym(3)$ under the canonical projection $G\rightarrow G/\zeta G\cong\Sym(3)$. Then $A$ is an abelian subgroup of index $2$ in $G$, and we have $G=A\rtimes G_{\omega}$ as well as
\[
A=(\zeta G)_2\times G_3=(\zeta G)_2\times(\zeta G)_3\times{\IZ/3\IZ}=\zeta G\times{\IZ/3\IZ};
\]
to see that $G_3=(\zeta G)_3\times{\IZ/3\IZ}$, follow the argument in \cite[proof of Proposition 3.3.3]{Bor20a}. Consequently, every automorphism of $\zeta G$ extends to an element of $\Aut_{\perm}(G)$, whence $3\geqslant\maol_{\perm}(G)\geqslant\maol(\zeta G)$, which in view of \cite[Theorem 1.1(1)]{Bor20a} implies that $|\zeta G|\in\{1,2,3,4,6\}$. It follows that
\[
(\deg(G),|G|) \in \{(3,6),(6,12),(9,18),(12,24),(18,36)\}.
\]
We conclude this proof by noting that according to the library of finite transitive permutation groups of small degree \cite{GAP4,TransGrp}, the only nonregular finite transitive permutation groups having one of these degree-size combinations as well as maximum normaliser orbit length $3$ are the groups $\D_{2n}\leqslant\Sym(n)$ for $n\in\{3,6\}$. \qedhere
\end{enumerate}
\end{proof}

\section{Proof of Theorem \ref{mainTheo}(2)}\label{sec4}

This is easy modulo the work from \cite[Section 4]{Bor20a}. As in \cite[Definition 4.1]{Bor20a}, denote by $G_n$ the finite $2$-group given by the following (power-commutator) presentation:
\begin{align*}
\langle x_1,\ldots,x_{2^n+1},a,b \mid  \,\,\,&[a,b]=[x_i,a]=[x_i,b]=1,[x_{2i-1},x_{2i}]=a,[x_{2i},x_{2i+1}]=b, \\
&[x_i,x_j]=1\text{ if }|i-j|>1,x_1^2=x_{2^n+1}^2=b, \\
&a^2=b^2=x_i^2=1\text{ if }1<i<2^n+1\rangle.
\end{align*}
We note the following known facts about the group $G_n$:
\begin{enumerate}
\item The order of $G_n$ is $2^{2^n+3}$, see \cite[Remark 4.2(1)]{Bor20a}.
\item The centre of $G_n$ is $\langle a,b\rangle$ and is of order $4$, see \cite[Remark 4.2(1)]{Bor20a}.
\item The group $\Aut_{\cent}(G_n)$ of central automorphisms of $G_n$ acts transitively on each nontrivial coset of $\zeta G_n$ in $G_n$, see \cite[beginning of the proof of Proposition 4.3]{Bor20a}.
\item $\Aut(G_n)=\Aut_{\cent}(G_n)\cup\Aut_{\cent}(G_n)\alpha_n$, where $\alpha_n$ is the automorphism of $G_n$ given by $a\mapsto a$, $b\mapsto b$, $x_i\mapsto x_i$ for $i\not=2^n$ and $x_{2^n}\mapsto x_{2^n}x_{2^n+1}$, see \cite[Remark 4.2(3) and Proposition 4.3]{Bor20a}.
\end{enumerate}
Now, consider the subgroup $H_n:=\langle x_{2^n}\rangle\cong\IZ/2\IZ$ of $G_n$, and view $G_n$ as a transitive permutation group via its action by right multiplication on the right cosets of $H_n$. Observe that
\[
H_n^{G_n}=\{H_n,\langle x_{2^n}a\rangle,\langle x_{2^n}b\rangle,\langle x_{2^n}ab\rangle\}=H_n^{\Aut_{\cent}(G_n)}.
\]
Therefore, using Lemma \ref{autPermLem} and fact (4) above, we find that $\Aut_{\perm}(G_n)=\Aut_{\cent}(G_n)$. In view of this and facts (2) and (3) from above, we conclude that $\maol_{\perm}(G_n)=4$. Since this holds for all $n\in\IN^+$, the statement of Theorem \ref{mainTheo}(2) now follows from fact (1).

\section{Proof of Theorems \ref{mainTheo}(3) and \ref{improvedBoundTheo}}\label{sec5}

We start with the proof of Theorem \ref{improvedBoundTheo}, which is easy modulo the following slightly weaker and reformulated version of Ledermann and Neumann's result \cite[Theorem 6.6]{LN56a}:

\begin{lemma}\label{lnLem}
Let $G$ be a finite group, and assume that $|\Aut(G)|\leqslant n$. Then $|G|\leqslant n^{n(1+\lfloor\log_2{n}\rfloor)}+1$.
\end{lemma}

\begin{proof}
Otherwise, we have
\[
|G|\geqslant n^{n(1+\lfloor\log_2{n}\rfloor)}+2\geqslant f(n+1),
\]
where $f$ is as in \cite[Theorem 6.6]{LN56a}. Hence by \cite[Theorem 6.6]{LN56a}, we have $|\Aut(G)|\geqslant n+1=|\Aut(G)|+1$, a contradiction.
\end{proof}

\begin{proof}[Proof of Theorem \ref{improvedBoundTheo}]
Let $\vec{g}$ be an arbitrary but fixed generating $d$-tuple of $G$. Consider the orbit $\Ocal:=\vec{g}^{\Aut(G)}$ of $\vec{g}$ under the (entry-wise) action of $\Aut(G)$. Combining the facts that the action of $\Aut(G)$ on the set of generating $d$-tuples of $G$ is semiregular and that all orbits of $\Aut(G)$ on $G$ are of length at most $c$, we find that
\[
|\Aut(G)|=|\Ocal|\leqslant c^d.
\]
The result now follows from Lemma \ref{lnLem}, applied with $n:=c^d$.
\end{proof}

Now we turn to the proof of Theorem \ref{mainTheo}(3). The meat of this proof lies in the following lemma:

\begin{lemma}\label{ledNeuLem}
Let $G$ be a finite transitive permutation group. Then, recalling the definition of the function $\ffrak$ from Notation \ref{mainNot}, we have
\[
|G|\leqslant\ffrak(d(G),|\Aut_{\perm}(G)|).
\]
\end{lemma}

As mentioned at the end of Section \ref{sec1}, our proof of Lemma \ref{ledNeuLem} is a modification of the short and self-contained proof due to Sambale \cite{Sam19a} that the order of an abstract finite group $G$ is bounded in terms of $|\Aut(G)|$. In contrast to Ledermann and Neumann's proof of their explicit result \cite[Theorem 6.6]{LN56a}, Sambale's argument first focuses on the commutator subgroup $G'$ and the abelianisation $G/G'$ and only turns to structural considerations concerning $\zeta G$ at its end. Unfortunately, it seems unclear whether Ledermann and Neumann's argument, which provides better explicit bounds than Sambale's proof, can be adapted to transitive permutation groups. Note that the main idea of Ledermann and Neumann's proof is to construct central automorphisms of $G$ by extending suitable automorphisms of $\zeta G$. More precisely, these automorphisms of $\zeta G$ must centralise $X\cap\zeta G$ where $X$ is a subgroup of $G$ mapping onto $G/\zeta G$ under the projection $G\rightarrow G/\zeta G$, and $X$ is chosen such that $X\cap\zeta G$ is of bounded order. The additional condition that the constructed central automorphisms shall map the (core-free) point stabiliser $G_{\omega}$ to a conjugate translates into additional mapping conditions on their restrictions to $\zeta G$, which do not seem to be well-controlled.

For our proof of Lemma \ref{ledNeuLem}, we will need the following concepts, which are from \cite[Definitions 5.6 and 5.8]{Bor20a}:

\begin{definition}\label{standardDef}
Consider the following concepts.
\begin{enumerate}
\item Let $p$ be a prime, and let $P$ be a finite abelian $p$-group. For $n\in\IN^+$ with $n\geqslant d(P)$, a \emph{length $n$ standard generating tuple of $P$} is an $n$-tuple $(x_1,\ldots,x_n)\in P^n$ such that $(x_1,\ldots,x_{d(P)})$ is a basis of $P$ (in the sense of Definition \ref{abelianBasisDef}) and $x_i=1_P$ for $i=d(P)+1,\ldots,n$.
\item Let $H$ be a finite abelian group. A \emph{standard generating tuple of $H$} is a $d(H)$-tuple $\vec{h}\in H^{d(H)}$ such that for each prime divisor $p$ of $|H|$, the entry-wise projection of $\vec{h}$ to the Sylow $p$-subgroup $H_p$ of $H$ is a standard generating tuple of $H_p$.
\item Let $G$ be a finite group. A \emph{standard tuple in $G$} is a $d(G/G')$-tuple with entries in $G$ and whose entry-wise image under the canonical projection $G\rightarrow G/G'$ is a standard generating tuple of $G/G'$.
\end{enumerate}
\end{definition}

\begin{definition}\label{pacDef}
Let $G$ be a finite group, let $n:=d(G/G')$, and let $\vec{g}=(g_1,\ldots,g_n)$ be a standard tuple in $G$.
\begin{enumerate}
\item The \emph{power-automorphism-commutator tuple associated with $\vec{g}$} is the $(2n+{n\choose 2})$-tuple
\[
(\pi_1,\ldots,\pi_n,\alpha_1,\ldots,\alpha_n,\gamma_{1,1},\gamma_{1,2},\ldots,\gamma_{1,n},\gamma_{2,3},\gamma_{2,4},\ldots,\gamma_{2,n},\ldots,\gamma_{n-1,n})
\]
with entries in $G'\cup\Aut(G')$ such that
\begin{itemize}
\item $\pi_i=g_i^{\ord_{G/G'}(g_iG')}\in G'$ for $i=1,\ldots,n$,
\item $\alpha_i\in\Aut(G')$ is the automorphism induced through conjugation by $g_i$ for $i=1,\ldots,n$, and
\item $\gamma_{i,j}=[g_i,g_j]\in G'$ for $1\leqslant i<j\leqslant n$.
\end{itemize}
\item Two standard tuples in $G$ are called \emph{equivalent} if and only if they have the same associated power-automorphism-commutator tuple.
\end{enumerate}
\end{definition}

We note that if $G$ is a finite group and $\vec{g},\vec{h}$ are equivalent standard tuples in $G$, then there is an $\alpha\in\C_{\Aut(G)}(G')$ such that $\vec{h}=(\vec{g})^{\alpha}$, see \cite[Remark 5.9]{Bor20a}. Another concept that will appear in the proof of Lemma \ref{ledNeuLem} is the following:

\begin{definition}\label{powerAutDef}
Let $G$ be a group. A \emph{power map automorphism of $G$} is an automorphism of $G$ of the form $g\mapsto g^e$ for all $g\in G$ and a fixed $e\in\IZ$. The power map automorphisms of $G$ form a subgroup of $\Aut(G)$ denoted by $\Aut_{\pow}(G)$.
\end{definition}

Note that this concept is distinct from the more general one of a \emph{power automorphism of $G$}, i.e., an automorphism of $G$ that stabilises every subgroup of $G$. Moreover, observe that if $G$ is a finite abelian group, then since $G$ has a cyclic direct factor of order $\Exp(G)$, we have
\[
\Aut_{\pow}(G)\cong\Aut_{\pow}(\IZ/\Exp(G)\IZ)=\Aut(\IZ/\Exp(G)\IZ).
\]

\begin{proof}[Proof of Lemma \ref{ledNeuLem}]
For notational simplicity, we will set $d:=d(G)\in\IN$ and $n:=|\Aut_{\perm}(G)|\in\IN^+$. Moreover, we fix a point $\omega\in\Omega$ and consider the associated point stabiliser $G_{\omega}\leqslant G$. We proceed in several steps.

First, we note that
\begin{equation}\label{gPrimeEq}
|G'|\leqslant n^{2n^3},
\end{equation}
which is the conclusion of \cite[Lemma 2]{Sam19a}. Reading through the proof of \cite[Lemma 2]{Sam19a}, we find that it only uses the assumption that $|\Inn(G)|\leqslant n$, which also holds in our case.

Next, we will show that each prime divisor $p$ of $|G|$ is at most $n+1$. We follow the proof of \cite[Lemma 3]{Sam19a}. If $|G/\zeta G|_p\not=1$, then $p\leqslant n$, since $\Inn(G)\leqslant\Aut_{\perm}(G)$. Otherwise, $|\zeta G|_p=|G|_p$ and $G=(\zeta G)_p\times Q$ by \cite[Theorem 3.3.1]{KS04a}, where $Q$ is a $p'$-group. Since $G_{\omega}$ is core-free in $G$, we have $G_{\omega}\cap(\zeta G)_p=\{1_G\}$, and thus $G_{\omega}\leqslant Q$ by the coprimality of $|(\zeta G)_p|$ and $|Q|$. Hence every automorphism of $(\zeta G)_p$ extends to an automorphism of $G$ in $\Aut_{\perm}(G)$, whence $n\geqslant|\Aut((\zeta G)_p)|\geqslant p-1$, as required.

Our next goal is to derive a certain upper bound on $\Exp(G/G')$. Observe that $\Exp(G/G')=\prod_p{\Exp((G/G')_p)}$ where $p$ ranges over the prime divisors of $|G:G'|$ and $(G/G')_p$ denotes the Sylow $p$-subgroup of $G/G'$. Since the number of factors in this product is at most $n+1$ by the previous paragraph, we find that for a suitable prime divisor $p$ of $|G:G'|$, we have $\Exp((G/G')_p)\geqslant\Exp(G/G')^{1/(n+1)}$. By the remark preceding this proof and the structure of automorphism groups of cyclic groups, this implies that $\Aut_{\pow}(G/G')$ contains an element $\beta$ of order at least
\[
\frac{1}{2}\phi(\Exp((G/G')_p)) \geqslant \frac{1}{4}\Exp((G/G')_p)^{1/2} \geqslant \frac{1}{4}\Exp(G/G')^{1/(2n+2)};
\]
for the lower bound on $\phi(\Exp((G/G')_p))$, see \cite[Lemma 2.4]{FS89a}. Fix a standard generating tuple $\vec{a}$ of $G/G'$ and consider the orbit
\[
\Ocal:=(\vec{a})^{\langle \beta\rangle}.
\]
Since the action of $\Aut(G/G')$ on generating tuples is semiregular, we find that
\[
|\Ocal|=\ord(\beta)\geqslant\frac{1}{4}\Exp(G/G')^{1/(2n+2)}.
\]
For each $k=0,1,\ldots,\ord(\beta)-1$, fix a lift $\vec{g}_k$ of $(\vec{a})^{\beta^k}$ in $G$ and collect those lifts in a set
\[
\Ocal':=\{\vec{g}_k \mid k=0,1,\ldots,\ord(\beta)-1\}.
\]
By definition, $\Ocal'$ is a set of standard tuples in $G$ of size
\[
|\Ocal'|=|\Ocal|=\ord(\beta)\geqslant\frac{1}{4}\Exp(G/G')^{1/(2n+2)}.
\]
By Formula (\ref{gPrimeEq}) and since a finite group of order $o$ has at most $o^{\log_2{o}}$ automorphisms (see e.g.~\cite[Lemma 5.5]{Bor20a}), the number of equivalence classes of standard tuples in $G$ is at most
\[
|G'|^d\cdot|\Aut(G')|^d\cdot|G'|^{{d\choose 2}} \leqslant n^{n^3d(1+d+4n^3\log_2{n})}.
\]
It follows that there is a nonempty subset $\Ocal_{\ast}\subseteq\Ocal'$ consisting of pairwise equivalent standard tuples in $G$ and such that
\[
|\Ocal_{\ast}|\geqslant\frac{\Exp(G/G')^{1/(2n+2)}}{4n^{n^3d(1+d+4n^3\log_2{n})}}.
\]
Fix $\vec{g}\in\Ocal_{\ast}$, and denote by $A$ the set of all $\alpha\in\C_{\Aut(G)}(G')$ such that $\vec{g}^{\alpha}\in\Ocal_{\ast}$. Then
\[
|A|=|\Ocal_{\ast}|\geqslant\frac{\Exp(G/G')^{1/(2n+2)}}{4n^{n^3d(1+d+4n^3\log_2{n})}},
\]
and the automorphisms in $A$ induce pairwise distinct automorphisms from $\langle\beta\rangle$ on $G/G'$. We claim that there is an $\alpha\in A$ such that the automorphism $\tilde{\alpha}$ of $G/G'$ induced by $\alpha$ satisfies
\[
\ord(\tilde{\alpha})\geqslant|\Ocal_{\ast}|^{1/2}\geqslant\frac{\Exp(G/G')^{1/(4n+4)}}{2n^{\frac{1}{2}n^3d(1+d+4n^3\log_2{n})}}.
\]
Indeed, this is clear if $|\Ocal_{\ast}|=1$ (where we may choose $\alpha=\id_G$), so assume that $|\Ocal_{\ast}|>1$. If $\ord(\tilde{\alpha})<|\Ocal_{\ast}|^{1/2}$ for all $\alpha\in A$, then since the cyclic group $\langle\beta\rangle$ has at most $k$ elements of order $k$ for each $k\in\IN^+$, we find that
\[
|\Ocal_{\ast}|>1+2+\cdots+\lfloor|\Ocal_{\ast}|^{1/2}\rfloor\geqslant|A|=|\Ocal_{\ast}|,
\]
a contradiction. We now know that there is an automorphism $\alpha\in\Aut(G)$ with the following properties:
\begin{itemize}
\item $\alpha$ centralises $G'$.
\item The automorphism $\tilde{\alpha}$ of $G/G'$ induced by $\alpha$ is a power map automorphism of $G/G'$.
\item $\ord(\alpha)\geqslant\ord(\tilde{\alpha})\geqslant\frac{\Exp(G/G')^{1/(4n+4)}}{2n^{\frac{1}{2}n^3d(1+d+4n^3\log_2{n})}}$.
\end{itemize}
Note that $\alpha$ does not necessarily stabilise $G_{\omega}$. However, $\tilde{\alpha}$, being a power map automorphism of $G/G'$, stabilises the projection of $G_{\omega}$ to $G/G'$. Therefore, each image of $G_{\omega}$ under an iterate of $\alpha$ has a generating tuple of the form $(g_1c_1,\ldots,g_tc_t,z_1,\ldots,z_u)$ where
\begin{itemize}
\item $t$ is the minimum number of generators of the projection of $G_{\omega}$ to $G/G'$,
\item $(g_1,\ldots,g_t)$ is a fixed lift in $G$ of a standard generating tuple of the projection of $G_{\omega}$ to $G/G'$,
\item $(z_1,\ldots,z_u)$ is a fixed generating tuple of $G_{\omega}\cap G'$, and
\item $(c_1,\ldots,c_t)$ is a variable $t$-tuple of elements of $G'$.
\end{itemize}
It follows that the length $\ell$ of the orbit of $G_{\omega}$ under $\langle\alpha\rangle$ satisfies
\[
\ell\leqslant|G'|^t\leqslant|G'|^d\leqslant n^{2n^3d}.
\]
Set $\gamma:=\alpha^{\ell}$. Then $\gamma$ stabilises $G_{\omega}$, whence $\gamma\in\Aut_{\perm}(G)$, and
\[
n=|\Aut_{\perm}(G)|\geqslant\ord(\alpha^{\ell})\geqslant\frac{\ord(\alpha)}{\ell}\geqslant\frac{\Exp(G/G')^{1/(4n+4)}}{2n^{\frac{1}{2}n^3d(5+d+4n^3\log_2{n})}}.
\]
This implies that
\[
\Exp(G/G')\leqslant 16^{n+1}n^{2n^3d(5+d+4n^3\log_2{n})+4n+4},
\]
which is the desired upper bound on $\Exp(G/G')$.

We can now conclude the proof as follows: Note that $G/G'$ is abelian, and since $G$ is $d$-generated, so is $G/G'$. It follows that
\[
|G:G'|\leqslant\Exp(G/G')^d\leqslant 16^{(n+1)d}n^{2n^3d^2(5+d+4n^3\log_2{n})+4nd+4d},
\]
and thus, using Formula (\ref{gPrimeEq}),
\[
|G|=|G'|\cdot|G:G'|\leqslant n^{2n^3}\cdot 16^{(n+1)d}n^{2n^3d^2(5+d+4n^3\log_2{n})+4nd+4d}=\ffrak(d,n),
\]
as required.
\end{proof}

Now that we have verified Lemma \ref{ledNeuLem}, we are ready to prove Theorem \ref{mainTheo}(3).

\begin{proof}[Proof of Theorem \ref{mainTheo}(3)]
As in the proof of Theorem \ref{improvedBoundTheo}, we have $|\Aut_{\perm}(G)|\leqslant c^d$, and the result follows from Lemma \ref{ledNeuLem}, using the monotonicity of the function $\ffrak$ in its second variable.
\end{proof}

\section{Proof of Theorem \ref{mainTheo}(4)}\label{sec6}

The proof is by contradiction -- let $G\leqslant\Sym(\Omega)$ be an insoluble finite transitive permutation group, and assume that $\maol_{\perm}(G)\leqslant23$. In particular, the maximum conjugacy class length in $G$ is at most $23$. Therefore, the arguments from \cite[Section 6]{Bor20a} show that $\zeta G=\Rad(G)$, that $G'\cong\Alt(5)$ and that $G=\zeta G\times G'$. Fix a point $\omega\in\Omega$, and consider the point stabiliser $G_{\omega}\leqslant G$. Since $G_{\omega}$ is core-free in $G$, we have $G_{\omega}\cap\zeta G=\{1_G\}$, so that $G_{\omega}$ embeds into $G'\cong\Alt(5)$ via the canonical projection $G\rightarrow G/\zeta G$. Denote by $P$ the image of $G_{\omega}$ under the canonical projection $G\rightarrow G/G'$, which we may and will view as a subgroup of $\zeta G$. We now note three important facts:
\begin{enumerate}
\item $\C_{\Aut(\zeta G)}(P)$ is trivial. Assume otherwise. Note that $\C_{\Aut(\zeta G)}(P)\times\Inn(G')$ embeds naturally into $\Aut_{\perm}(G)$ via the injective group homomorphism
\begin{align*}
&\iota: (\alpha,\beta) \mapsto ((z,c) \mapsto (z^{\alpha},c^{\beta})) \\
&\text{ for all }\alpha\in\C_{\Aut(\zeta G)}(P),\beta\in\Inn(G'),z\in\zeta G,c\in G',
\end{align*}
which has the property that if $\beta$ is the conjugation by $c_0\in G'$, then $G_{\omega}^{\iota(\alpha,\beta)}=G_{\omega}^{c_0}$, a $G$-conjugate of $G_{\omega}$. Therefore, and since $G'\cong\Alt(5)$ has a conjugacy class of length $20$, it follows that $\maol_{\perm}(G)\geqslant 2\cdot 20=40>23$, a contradiction.
\item $|P|>1$. Otherwise, we have $G_{\omega}\leqslant G'\cong\Alt(5)$, and since there is exactly one conjugacy class of subgroups of $G'$ that are isomorphic to $G_{\omega}$, we find that $\Aut(G')$ embeds naturally into $\Aut_{\perm}(G)$ by Lemma \ref{autPermLem}, whence $\maol_{\perm}(G)\geqslant\maol(G')=\maol(\Alt(5))=24>23$, a contradiction.
\item $P$ is a quotient of $G_{\omega}/G_{\omega}'$. This is clear since $P$ is by definition an abelian quotient of $G_{\omega}$.
\end{enumerate}
We will now go through the finitely many possible (abstract group) isomorphism types of $G_{\omega}$ and reduce the proof of Theorem \ref{mainTheo}(4) to checking a finite list of possibilities for $(G,G_{\omega})$, which are listed in Table \ref{23table} below.
\begin{enumerate}
\item Case: $G_{\omega}\cong\Alt(5)$ or $|G_{\omega}|=1$. Then $G_{\omega}/G_{\omega}'$ is trivial, whence $P$ is trivial by fact (3) above. However, this contradicts fact (2).
\item Case: $G_{\omega}\cong\Alt(4)$ or $G_{\omega}\cong\IZ/3\IZ$. Then $G_{\omega}/G_{\omega}'\cong\IZ/3\IZ$, whence $P\cong\IZ/3\IZ$ by facts (2) and (3) above. Using fact (1) and Lemma \ref{centAbelianLem}, it follows that $\zeta G$ is cyclic of order $3$ or $6$. Hence, up to abstract group isomorphism, the pair $(G,G_{\omega})$ is one of the possibilities listed in rows 1--4 of Table \ref{23table}.
\item Case: $G_{\omega}\cong\IZ/5\IZ$. Then $G_{\omega}/G_{\omega}'\cong\IZ/5\IZ$, whence $P\cong\IZ/5\IZ$ by facts (2) and (3) above. Using fact (1) and Lemma \ref{centAbelianLem}, it follows that $\zeta G$ is cyclic of order $5$ or $10$. Hence, up to abstract group isomorphism, the pair $(G,G_{\omega})$ is one of the possibilities listed in rows 5 and 6 of Table \ref{23table}.
\item Case: $G_\omega$ is isomorphic to one of $\D_{10}$, $\Sym(3)$ or $\IZ/2\IZ$. Then $G_{\omega}/G_{\omega}'\cong\IZ/2\IZ$, whence $P\cong\IZ/2\IZ$ by facts (2) and (3) above. Using fact (1) and Lemma \ref{centAbelianLem}, it follows that $\zeta G=P\cong\IZ/2\IZ$. Hence, up to abstract group isomorphism, the pair $(G,G_{\omega})$ is one of the possibilities listed in rows 7--9 of Table \ref{23table}.
\item Case: $G_\omega\cong(\IZ/2\IZ)^2$. Then $G_{\omega}/G_{\omega}'\cong(\IZ/2\IZ)^2$, whence by facts (2) and (3) above, either $P\cong\IZ/2\IZ$ or $P\cong(\IZ/2\IZ)^2$. In either case, we have $\zeta G=P$ by fact (1) and Lemma \ref{centAbelianLem}. Therefore, up to abstract group isomorphism, the pair $(G,G_{\omega})$ is one of the possibilities listed in rows 10 and 11 of Table \ref{23table}.
\end{enumerate}
We now give Table \ref{23table}, which not only lists the $11$ remaining possibilities for $(G,G_{\omega})$ extracted from the above arguments, but also specifies the value of $\maol_{\perm}(G)$ in each case, which was computed using GAP \cite{GAP4}. Since $\maol_{\perm}(G)>23$ throughout, the proof of Theorem \ref{mainTheo}(4) is complete.

\begin{table}[h]
\begin{center}
\begin{tabular}{|c|c|c|c|}
\hline
row no. & $G$ & $G_{\omega}$ & $\maol_{\perm}(G)$ \\ \hline
1 & $\IZ/3\IZ\times\Alt(5)$ & $\langle(\overline{1},(1,2,3))\rangle$ & $48$ \\ \hline
2 & $\IZ/6\IZ\times\Alt(5)$ & $\langle(\overline{2},(1,2,3))\rangle$ & $48$ \\ \hline
3 & $\IZ/3\IZ\times\Alt(5)$ & $\langle(\overline{0},(1,2)(3,4)),(\overline{1},(1,2,3))\rangle$ & $40$ \\ \hline
4 & $\IZ/6\IZ\times\Alt(5)$ & $\langle(\overline{0},(1,2)(3,4)),(\overline{2},(1,2,3))\rangle$ & $40$ \\ \hline
5 & $\IZ/5\IZ\times\Alt(5)$ & $\langle(\overline{1},(1,2,3,4,5))\rangle$ & $80$ \\ \hline
6 & $\IZ/10\IZ\times\Alt(5)$ & $\langle(\overline{2},(1,2,3,4,5))\rangle$ & $80$ \\ \hline
7 & $\IZ/2\IZ\times\Alt(5)$ & $\langle(\overline{1},(1,2)(3,4))\rangle$ & $24$ \\ \hline
8 & $\IZ/2\IZ\times\Alt(5)$ & $\langle(\overline{0},(1,2,3)),(\overline{1},(2,3)(4,5))\rangle$ & $24$ \\ \hline
9 & $\IZ/2\IZ\times\Alt(5)$ & $\langle(\overline{0},(1,2,3,4,5)),(\overline{1},(2,5)(3,4))\rangle$ & $24$ \\ \hline
10 & $\IZ/2\IZ\times\Alt(5)$ & $\langle(\overline{0},(1,2)(3,4)),(\overline{1},(1,3)(2,4))\rangle$ & $24$ \\ \hline
11 & $(\IZ/2\IZ)^2\times\Alt(5)$ & $\langle((\overline{1},\overline{0}),(1,2)(3,4)),((\overline{0},\overline{1}),(1,3)(2,4))\rangle$ & $72$ \\ \hline
\end{tabular}
\end{center}
\caption{The remaining possibilities for $(G,G_{\omega})$}
\label{23table}
\end{table}

\section{Concluding remarks}\label{sec7}

We conclude this paper with some related open questions for further research. Firstly, as mentioned in Section \ref{sec5}, our Lemma \ref{ledNeuLem} is a partial generalisation (from finite abstract groups to finite transitive permutation groups) of a celebrated theorem of Ledermann and Neumann, \cite[Theorem 6.6]{LN56a}. The following question asks whether Ledermann and Neumann's theorem can be extended to transitive permutation groups in its full strength:

\begin{question}\label{ledNeuQues}
Is there an (explicit) function $f:\IN^+\rightarrow\IN^+$ such that for every finite transitive permutation group $G$, one has $|G|\leqslant f(|\Aut_{\perm}(G)|)$?
\end{question}

Another interesting question is whether Robinson and Wiegold's structural characterisation \cite[Theorem 1]{RW84a} can be extended to transitive permutation groups:

\begin{question}\label{robWieQues}
Let $G$ be a (not necessarily finite) transitive permutation group. Is it true that the following are equivalent?
\begin{enumerate}
\item The supremum of the $\Aut_{\perm}(G)$-orbit sizes on $G$ is finite.
\item The torsion subgroup $T$ of $\zeta G$ is finite and $\Aut_{\perm}(G)$ induces a finite group of automorphisms in $G/T$.
\end{enumerate}
\end{question}

It would also be interesting to investigate to what extent Theorem \ref{mainTheo} can be generalised to arbitrary (not necessarily transitive) permutation groups $G$ of finite degree. For example, even under the assumption that $\maol_{\perm}(G)=1$, there are infinitely many such $G$ up to permutation group isomorphism (any finite-degree permutation group $G$ of order $2$ is an example), but the following is still an interesting question:

\begin{question}\label{intransQues}
Can the finite-degree permutation groups $G$ with $\maol_{\perm}(G)\leqslant 3$ be classified, and are they of bounded order?
\end{question}

We finish with the following question about extending the results of \cite{Bor20a} to another natural setting.

\begin{question}\label{GLQues}
Let $G\leqslant \GL(d,q)$ and let $\Aut_{\mathrm{linear}}(G)$ be the subgroup of $\Aut(G)$ induced by the conjugation action of $\N_{\GL(d,q)}(G)$ on $G$. Is it possible to classify all groups $G$ for which all orbits of $\Aut_{\mathrm{linear}}(G)$ have length at most three?
\end{question}

\end{document}